\documentclass[11pt,times,letter]{article}

\usepackage{amsmath, amsfonts,amsthm,amssymb,bm}
\usepackage{dsfont}
\usepackage{extarrows}

\usepackage[english]{babel}

\usepackage[normalem]{ulem}

\usepackage{latexsym,amssymb,amsfonts,graphicx}
\usepackage{amsmath,dsfont}
\usepackage{verbatim}
\usepackage{mathrsfs}
\usepackage{bm}
\usepackage{color}
\usepackage{epsfig}

\usepackage{url}
\usepackage{bm}
\usepackage{natbib}

\usepackage[colorlinks,linkcolor=blue,citecolor=blue,anchorcolor=blue]{hyperref}

\newcommand{\Px}{ \mathbb{P} }

\newcommand{\N}{ \mathbb{N} }
\newcommand{\Ex}{ \mathbb{E} }

\def\esssup_#1{\underset{#1}{\Xi}}
\def\essinf_#1{\underset{#1}{\mathrm{ess\,inf\, }}}
\def\argmax_#1{\underset{#1}{\mathrm{arg\,max\, }}}
\def\argmin_#1{\underset{#1}{\mathrm{arg\,min\, }}}

\newcommand{\Fx}{\mathbb{F} }

\newcommand{\F}{\mathcal{F}}

\newcommand{\R}{\mathds{R}}

\newcommand{\Asa}{{\bf(A$_{s1}$)}}
\newcommand{\Asb}{{\bf(A$_{s2}$)}}
\newcommand{\Asc}{{\bf(A$_{s3}$)}}

\newtheorem{theorem}{Theorem}[section]
\newtheorem{definition}{Definition}[section]

\newtheorem{proposition}[theorem]{Proposition}
\newtheorem{remark}[theorem]{Remark}
\newtheorem{lemma}[theorem]{Lemma}

\allowdisplaybreaks[4]

\definecolor{Red}{rgb}{1.00, 0.00, 0.00}

\definecolor{DRed}{rgb}{0.5, 0.00, 0.00}

\definecolor{Blue}{rgb}{0.00, 0.00, 1.00}

\definecolor{Green}{rgb}{0.0, 0.4, 0.0}

\begin{document}

\title{Approximating Nash Equilibrium for Production Control with Sticky Price}
\author{Chunmei Jiang \thanks{Email: meizi81@aliyun.com, Xi'an Jiaotong University City College, No. 8715, Shangji Road, Xi'an, China.}
\and
Tongqing Li \thanks{Email: ltqing@mail.ustc.edu.cn, School of Mathematical Sciences, University of Science and Technology of China, Hefei, China.}
\and
Jie Yu \thanks{Email: yj123456@mail.ustc.edu.cn, School of Mathematical Sciences, University of Science and Technology of China, Hefei, China.}
}

\maketitle

\begin{abstract}
  We study a mean field game problem arising from the production control for multiple firms with price stickiness in the commodity market. The price dynamics for each firm is described as a (controlled) jump-diffusion process with mean-field interaction. Each firm aims to maximize her expectation of cumulative net profit coupled with each other through price processes. By solving the limiting control problem and a fixed-point problem, we construct an explicit approximating Nash equilibrium when the number of firms grows large.
  \vspace{0.2 cm}

  \vspace{0.2 cm}

\end{abstract}

\section{Introduction}

In this paper, we consider an optimal production control problem with price stickiness for a large number of firms in the commodity market. The objective of each firm is to seek an optimal adjustment strategy of the production rate in order to maximize its expectation of the overall net profit. This type of the production planning problem has been arisen in many fields such as the electricity generation planning (\cite{AminloeiGhaderi10}) and optimal investment (\cite{GuoPham05}). Our problem is related to that considered in \cite{GuoPham05} which study a partially reversible investment problem in which each firm can adjust its production capacity according to market fluctuations. In this article, the output rate of each firm is formulated as a controlled geometry Brownian motion, i.e., for $i=1,\ldots,n$, the production output rate process $X^{i,u^i}=(X_t^{i,u^i})_{t\geq0}$ for firm $i$ evolves as follows: $X_0^{i,u^i}=X_0^i$,
\begin{align}\label{eq:Xi-state}
dX_t^{i,u^i}=X_t^{i,u^i}(-\mu_idt+\sigma_i dW_t^i)+u_t^idt,
\end{align}
where $-\mu_idt+\sigma_i dW_t^i$ denotes the random capital depreciation rate with $\mu_i,\sigma_i>0$, and $W^i=(W_t^i)_{t\geq0}$ for $i\in\N$ are independent Brownian motions. The firm $i$ can adjust its output rate via its output rate control process $u^i=(u^i_t)_{t\geq0}$. We will explain the model in Section \ref{sec:model}. As in \cite{AbelEberly97} and \cite{WangHuang19}, the objective (cost) functional of our production control problem is considered to be quadratic in the output rate control for each firm. This is different from the linear cost of adjustment considered in \cite{GuoPham05} which results in a singular control problem for a single agent. On the other hand, unlike the case in \cite{GuoPham05} that the profit function only depends on the production capacity, we propose the profit rate of firm $i$ at time $t$ which relies both on the current price $P_t$ of the common product and its production rate $X_t^{i,u^i}$.

We also incorporate price stickiness into our market model. The price stickiness is the resistance of the market price to change quickly, despite it is optimal from an economic perspective that the price should change instantaneously when supply and demand change. It has been proposed in many popular macro models.  Fershtman and Kamien \cite{FershtmanKamien87} consider a duoplistic competition model with sticky price and derive an explicit open-loop and closed-loop Nash equilibrium.  Cellini and Lambertini \cite{CelliniLambertini04} extend the results mentioned above to the case which has more than two players. They show that the firms prefer the open-loop equilibrium as the price of goods is higher than the equilibrium price corresponding to the close-loop Nash equilibrium strategies. However, in both of the above references, each firm only solves a one-dimensional optimal control problem when the control strategies of other firms are fixed because the corresponding (controlled) state process is only one-dimensional.  Wang and Huang \cite{WangHuang19} explore an output control problem with a large number of producers supplying a certain product and sticky price whose state process is two-dimensional and is more similar to our case. They give both Nash and social optimum strategies and further compare the two solutions numerically. We also stress that random losses of goods during transportation and storage may lead to the surging demand of the market and hence the discontinuity of paths of price process. Thus, differently from the case considered in the papers reviewed above, we employ Poisson processes to describe the occurrences of random losses in the modelling of the price process.

Since the objective functional of each form is coupled through the discontinuous price process, it is in general hard to derive the Nash equilibrium, a fortiori the explicit formula, especially when the number of firms grows large. To bypass this difficulty, we establish the mean field game (MFG), which is independently introduced in \cite{LasryLions07} and \cite{Huang-M-C06}. It provides a powerful methodology for deducing the computation complexity when the number of players is large by constructing an approximately optimal action via the limiting problem. Thus, it has been rapidly developed since its inception and has been used in many fields, such as Bo et al. \cite{Bo-W-Y21} on the optimal investment with contagious risk (\cite{BoCapponi16}), Carmona et al. \cite{Carmona-F-S15} for systemic risk control and  Lacker and Soret \cite{LackerSoret20} for the optimal investment with consumption. In order to establish MFG, we solve the optimal control problem for a so-called representative firm (i.e. the limiting control problem), and then establish the unique fixed point of the mapping related to the consistence condition.


The rest of this paper is organized as follows. Section \ref{sec:model} formulates the output adjustment problem with price stickiness in the case with finite firms. Section \ref{sec:MFG} studies the optimal control problem for the representative firm, and then the resulting fixed point problem is solved. In Section \ref{sec:approx-Nash}, we give an explicit strategy which can be shown to be an approximate Nash equilibrium of the MFG.

\section{Problem Formulation}\label{sec:model}

We consider a commodity market consisting of $n$ firms which produce the same goods. Each firm can adjust its own production capacity.

Let $(\Omega,\F,\Px)$ be a complete probability space with a reference filtration $\Fx=(\F_t)_{t\geq0}$ satisfying the usual conditions. This space supports $n$ independent Brownian motions $W^i=(W_t^i)_{t\geq0}$ for $i=1,\ldots,n$, and $n$ independent Poisson process $N^i=(N_t^i)_{t\geq0}$ with intensity parameter $\lambda_i>0$ for $i=1,\ldots,n$.
The production output rate process $X^{i,u^i}=(X_t^{i,u^i})_{t\geq0}$ of firm $i$ evolves as \eqref{eq:Xi-state}, i.e.,
\begin{equation*}
dX_t^{i,u^i}=X_t^{i,u^i}(-\mu_idt+\sigma_i dW_t^i)+u_t^idt,\quad X_0^{i,u^i}=X_0^i\in\R,
\end{equation*}
where, for $i=1,\ldots,n$,
\begin{itemize}
  \item $u_t^i$ is the output rate adjusted by firm $i$ at time $t$, which suffers a quadratic cost $r_i|u_t^i|^2$ with the cost parameter $r_i>0$.
  \item $\mu_i>0$ is the depreciation rate of the production of firm $i$.
  \item $\sigma_i>0$ is the volatility of the production output for firm $i$.
\end{itemize}

We introduce the price dynamics of the same goods depending on the state processes of all firms as follows: $P_0^u=p_0\in\R$,
\begin{equation}\label{eq:Price}
  dP_t^u=\alpha\bigg(\underbrace{\beta-\frac{1}{n}\sum_{i=1}^n X_t^{i,u^i}}_{\textrm{theoretical price}}-P_t^u\bigg)dt+\underbrace{\frac{\alpha }{n}\sum_{i=1}^{n}\gamma_iX_t^{i,u^i}dN_t^{i}}_{\textrm{random losses}},
\end{equation}
where $u=(u^1,\ldots,u^n)\in\mathbb{U}^n$ denotes the vector of strategies and $\mathbb{U}$ stands for the admissible control space which will be specified later. The term $\beta-\frac{1}{n}\sum_{i=1}^n X_t^{i,u^i}$ is referred to as the theoretical price derived from the linear inverse demand function, while $\beta>0$ is the constant demand rate. As in \cite{FershtmanKamien87} and \cite{WangHuang19}, the price adjusts (with the speed of adjustment $\alpha>0$) proportionally to the difference between the theoretical price and the current price. However, losses of goods may arise during transportation and storage, so that the output of each firm may not completely put into the market. We simply assume that random losses of the output of firm $i$, which occur at the sequence of jump times of $N^i=(N_t^i)_{t\geq0}$, is given by $\gamma_i X_t^{i,u^i}$ with the ratio parameter $\gamma_i>0$. Thus, the reward functional for firm $i$ is given by
\begin{equation}\label{eq:reward-func}
  R_i(u):=\Ex\left[\int_{0}^{\infty}e^{-\rho t}\left((1-c_i)P_t^uX_t^{i,u^i}-r_i(u_t^i)^2\right)dt\right],
\end{equation}
where $\rho>0$ is the discount rate and $c_i P_t X_t^{i,u^i}$ is the production cost for firm $i$ with ratio $c_i\in(0,1)$.

\begin{remark}
Our price process described as \eqref{eq:Price} may result in a negative price  with positive probability. In fact, in most commodity markets, the prices of goods such as oil, onions and electricity can fall below zero. The negative price may be caused by (i) the supply is larger than the demand; or (ii) the storage cost is too expensive when the storage is getting closer to the capacity. Similar setup with possible negative prices for goods has been  considered in \cite{CelliniLambertini04}, \cite{FershtmanKamien87} and \cite{WangHuang19}.
\end{remark}

The aim of firm $i$ is to maximize its reward functional $R_i(u)$ over $u^i\in\mathbb{U}$. The admissible control set $\mathbb{U}$ is defined to be the set of $\Fx$-adapted r.c.l.l. process $u=(u_t)_{t\geq0}$ such that
\begin{align}\label{eq:|u|inf}
    \|u\|_\rho:=\left\{\Ex\left[\int_0^\infty e^{-\rho t}|u_t|^2dt\right]\right\}^{\frac{1}{2}}<\infty.
\end{align}
Building upon the objective functional \eqref{eq:reward-func},  we give the definition of Nash equilibrium as follows:
\begin{definition}\label{def:Nash}
A policy $u^*=(u^{*,1},\ldots,u^{*,n})\in\mathbb{U}^n$ is called a Nash equilibrium for this $n$-player game, if for all $i=1,\ldots,n$,
\begin{align*}
  R_i(u^*)=\sup_{u^i\in\mathbb{U}}R_i(u^i,u^{*,-i}).
\end{align*}
Here the $(u^i,u^{*,-i})=(u^{*,1},\ldots,u^{*,i-1},u^i,u^{*,i+1},\ldots,u^{*,n})$. For a given $\varepsilon>0$, $u^*=(u^{*,1},\ldots,u^{*,n})\in\mathbb{U}^n$ is an $\varepsilon$-Nash equilibrium, if
\begin{align}\label{eq:NE00}
  R_i(u^*)\geq\sup_{u^i\in\mathbb{U}}R_i(u^i,u^{*,-i})-\varepsilon,~ \forall i=1,\ldots,n.
\end{align}
\end{definition}
When $n$ is large, the Nash equilibrium is hard to compute. As the impact of each single firm is insignificant, it is often convenient to study the optimal control problem for the representative firm (i.e., the limiting problem as $n\to\infty$) and establish an approximating Nash equilibrium. To do it, we impose the following assumptions throughout the paper:
\begin{description}
  \item[{\Asa}] The initial outputs $X_0^i$, $i=1,\ldots,n$, $(W^1,\ldots, W^n)$ and $(N^1,\ldots, N^n)$ are mutually independent. Moreover,  $\sup_{i\in\N}\Ex[|X_0^i|^2]<\infty$ and {\small$x_0:=\lim_{n\to\infty}\frac{1}{n}\sum_{i=1}^n\Ex[X_0^i]$} exists.
  \item[{\Asb}] The type vector $\theta_i:=(\mu_i,\sigma_i,\gamma_i,\lambda_i,r_i,c_i)\in\mathbb{R}_+^6$ converges to some $\theta=(\mu,\sigma,\gamma,\lambda,r,c)\in\mathbb{R}_+^6$ as $i\to\infty$. Here $\mathbb{R}_+:=(0,\infty)$.
  \item[{\Asc}] For all $i\in\N$, $\sigma_i^2<2\mu_i$ and $1-\gamma_i\lambda_i>0$.
\end{description}
The 1st condition in the assumption {\Asc} is used to guarantee that the uncontrolled ($u^i\equiv0$) process $X_t^{i,0}\to0$, as $t\to\infty$, $\Px$-a.s.. The 2nd condition can imply the nonnegativity of the expectation of the accumulate output of an arbitrary firm.

The following moment estimate will be used frequently in the forthcoming sections.
\begin{lemma}\label{lem:estimates}
Let assumptions {\Asa}-{\Asc} hold. Then, there exists $D_1>0$ independent of $i$ and the control $u^i$ such that
\begin{align*}
  \left\|X^{i,u^i}\right\|_\rho^2\leq D_1\left(1+\|u^i\|_\rho^2\right)\qquad \forall~ u^i\in\mathbb{U}.
\end{align*}
\end{lemma}

\begin{proof}
Define $Y^{i,u^i}_t=e^{-\frac{\rho t}{2}}X^{i,u^i}_t$. It\^{o}'s formula yields that
\begin{align*}
dY^{i,u^i}_t=Y^{i,u^i}_t\left[-(\mu_i+\rho/2)dt+\sigma_idW_t^i\right]+e^{-\frac{\rho t}{2}}u_t^idt.
\end{align*}
Since it is a linear SDE, it admits the following explicit solution: for $t\geq0$,
\begin{align*}
Y^{i,u^i}_t=e^{-\left(\mu_i+\frac{\rho}{2}+\frac{\sigma_i^2}{2}\right)t+\sigma_iW_t^i}X_0^i+\int_0^te^{-\frac{\rho s}{2}}u_s^ie^{-\left(\mu_i+\frac{\rho}{2}+\frac{\sigma_i^2}{2}\right)(t-s)+\sigma_i(W_t^i-W_s^i)} ds.
\end{align*}
Below, let $C>0$ be a generic constant independent of $i$ and the control $u^i$, but it will be different from line to line. Note that Brownian motion $W^i$ is independent of $X_0^i$, it follows from H\"{o}lder's inequality and {\Asa} that
\begin{align*}
  \Ex[|Y^{i,u^i}_t|^2]&\leq C\bigg\{e^{-(2\mu_i+\rho+\sigma_i^2)t}\Ex[|X_0^i|^2] \Ex\left[e^{2\sigma_iW_t^i}\right]\\
  &\quad+\int_0^t e^{-\rho s}|u_s^i|^2ds\times\int_0^t e^{-(2\mu_i+\rho+\sigma_i^2)(t-s)}\Ex\left[e^{2\sigma_i(W_t^i-W_s^i)}\right]ds \Bigg\}\\
  &\leq C\left[e^{-(2\mu_i+\rho-\sigma_i^2)t}+\|u^i\|_\rho^2\int_0^t e^{-(2\mu_i+\rho-\sigma_i^2)(t-s)}ds\right].
\end{align*}
Since the discount rate $\rho>0$, $2\mu_i+\rho-\sigma_i^2>2\mu_i-\sigma_i^2>0$ for all $i\geq 1$. By {\Asb}, there exists a constant $D_1>0$ independent of $i$ and $u^i$ s.t. $\|X^{i,u^i}\|_\rho^2\leq D_1(1+\|u^i\|_\rho^2)$.
Thus, we complete the proof of the lemma.
\end{proof}

A direct implication of Lemma \ref{lem:estimates} is $R_i(u)<\infty$ for all $u\in\mathbb{U}^n$.

\section{The Mean Field Game Problem}\label{sec:MFG}

This section aims to study the associated MFG problem. We first deal with the control problem for a so-called representative firm, and then solve a fixed-point problem raised by the consistence condition (c.f. \eqref{eq:consis} below).

\subsection{Optimal control for the representative firm}

In the context of MFG with \eqref{eq:Xi-state}-\eqref{eq:reward-func}, the control problem for the representative firm is described as follows: for a given $m^X=(m^X_t)_{t\geq0}\in C_{\rho/2}([0,\infty);\R)$ with
\begin{align}\label{eq:spaceCrho}
  C_{\rho/2}([0,\infty);\R)&:=\bigg\{f\in C([0,\infty);\R);~\exists~ \rho'\in[0,\rho)\textrm{s.t.}~\sup_{t\geq0}e^{-\frac{\rho' t}{2}}|f(t)|<\infty \bigg\},
\end{align}
and the type parameter $\theta=(\mu,\sigma,\gamma,\lambda,r,c)\in\mathbb{R}_+^6$, we consider the following stochastic control problem:
\begin{equation}\label{eq:repre-cont-pr}
\left\{\begin{aligned}
\sup_{u\in\mathbb{U}}R(u)
&=\sup_{u\in\mathbb{U}}\mathbb{E}\left[\int_{0}^{\infty}e^{-\rho t}\left((1-c)m^P_tX_t^u-ru_t^2\right)dt\right],\\[0.6em]
\textrm{s.t.}~dX_t^u&=X_t^u(-\mu dt+\sigma dW_t)+u_tdt,~ X_0^u=x_0\in\R,
\end{aligned}\right.
\end{equation}
where $W=(W_t)_{t\geq0}$ is an $\Fx$-Brownian motion under $(\Omega,\F,\Px)$, and $m^P=(m^P_t)_{t\geq0}$ satisfies the dynamics:
\begin{equation}\label{eq:lim-m^P}
  dm^P_t=\alpha[\beta-(1-\lambda\gamma)m^X_t-m^P_t]dt,\quad m_0^P=p_0>0.
\end{equation}
The value function for this representative agent is given by: for $(t,x)\in[0,\infty)\times\R$,
{\small\begin{equation}\label{eq:value-func}
V(t,x)=\sup_{u\in\mathbb{U}}\Ex_{t,x}\left[\int_{t}^{\infty}e^{-\rho(s-t)} \left((1-c)m^P_{s}X_{s}^u-ru_{s}^2\right)ds\right],
\end{equation}}
where $\Ex_{t,x}[\cdot]:=\Ex[\cdot|X_t=x]$. By the dynamic programming principle, the value function $V(t,x)$ formally satisfies the following HJB equation on $(t,x)\in[0,\infty)\times\R$:
\begin{align}\label{eq:HJB}
 \rho V&=
 \sup_{u\in\R}\bigg[\frac{\partial V}{\partial t}+(-\mu x+u)\frac{\partial V}{\partial x}+\frac{\sigma^2}{2}x^2\frac{\partial^2V}{\partial x^2}+(1-c)m^P_tx-ru^2\bigg].
\end{align}

We seek the following solution form for Eq.~\eqref{eq:HJB}:
\begin{align}\label{eq:Vtx}
   V(t,x)=g_t x+h_t,
\end{align}
where $g=(g_t)_{t\geq0}\in C_{\rho/2}([0,\infty);\R))\cap C^1([0,\infty);\R)$ and $h=(h_t)_{t\geq0}\in C_{\rho}([0,\infty);\R))\cap C^1([0,\infty);\R)$ which will be determined later. Note that the maximum in \eqref{eq:HJB} is attained at $\frac{1}{2r}\frac{\partial V}{\partial x}(t,x)=\frac{1}{2r}g_t$. We then plugging it into \eqref{eq:HJB} to have that $(g_t,h_t)$ satisfy that
\begin{align}\label{eq:ansatz-gh}
\begin{cases}
  \displaystyle \rho g_t=\frac{dg_t}{dt}-\mu g_t+(1-c)m^P_t,\\[0.8em]
\displaystyle  \rho h_t=\frac{dh_t}{dt}+\frac{1}{4r}g_t^2.
\end{cases}
\end{align}
The well-posedness of Eq.~\eqref{eq:ansatz-gh} is given in the following lemma:
\begin{lemma}\label{lem:wellpo-HJB}
Given $m^X\in C_{\rho/2}([0,\infty);\R)$, there exists a unique solution $(g,h)\in C_{\rho/2}([0,\infty),\R)\cap C^1([0,\infty);\R)\times C_{\rho}([0,\infty),\R)\cap C^1([0,\infty);\R)$ to Eq. \eqref{eq:ansatz-gh}. Moreover, we have
\begin{align}\label{eq:gh}
\begin{cases}
  \displaystyle g_t= (1-c)\int_{t}^{\infty}e^{-(\mu+\rho)(s-t)}m^P_{s}ds,\\[0.8em]
  \displaystyle h_t= \frac{1}{4r}\int_t^\infty e^{-\rho(s-t)}g_s^2 ds.
\end{cases}
\end{align}
\end{lemma}

\begin{proof}
We first show that $m^P$ defined by \eqref{eq:lim-m^P} belongs to $C_{\rho/2}([0,\infty);\R)$ whenever $m^X\in C_{\rho/2}([0,\infty);\R)$. Let $C>0$ be a generic constant independent of $t$, which will be different form line to line. By \eqref{eq:spaceCrho}, there exists $\rho'\in[0,\rho)$ such that $M(\rho'):=\sup_{t\geq0}e^{-\frac{\rho' t}{2}}|m_t^X|<\infty$. Then, by \eqref{eq:lim-m^P},
\begin{align*}
  e^{-\frac{\rho' t}{2}}|m_t^P|
  &\leq e^{-\frac{\rho' t}{2}}C\left(1+\int_0^t e^{-\alpha(t-s)}|m_s^X|ds\right)
  \leq C\left(1+\int_0^t e^{-\alpha(t-s)}e^{-\frac{\rho' s}{2}}|m_s^X|ds\right)\\
  &\leq C\left(1+M(\rho')\int_0^t e^{-\alpha(t-s)}ds\right)<\infty.
\end{align*}
This yields $m^P\in C_{\rho/2}([0,\infty);\R)$.

Observe that the 1st equation in \eqref{eq:ansatz-gh} is a first-order linear ODE. Its solution is given by: for $t\geq0$,
\begin{align*}
g_t=e^{(\mu+\rho)t}\left(-(1-c)\int_0^te^{-(\mu+\rho)s}m^P_{s}ds+\widetilde{D}_1\right),
\end{align*}
where $\widetilde{D}_1\in\R$ is a constant which will determined later. We next claim that $\widetilde{D}_1=(1-c)\int_0^\infty e^{-(\mu+\rho)s}m^P_{s}ds$, i.e.,
\begin{align}\label{eq:claim1}
    g_t=(1-c)\int_{t}^{\infty}e^{-(\mu+\rho)(s-t)}m^P_{s}ds,\quad t\geq0
\end{align}
is the unique solution to \eqref{eq:ansatz-gh} in $C_{\rho/2}([0,\infty);\R)\cap C^1([0,\infty);\R)$.
Otherwise, if $\widetilde{D}_1=(1-c)\int_0^\infty e^{-(\mu+\rho)s}m^P_{s}ds+\widetilde{D}_2$ for some constant $\widetilde{D}_2\neq0$, then $g$ can be expressed as:
\begin{align*}
g_t=(1-c)\int_{t}^{\infty}e^{-(\mu+\rho)(s-t)}m^P_{s}ds+\widetilde{D}_2e^{(\mu+\rho)t}.
\end{align*}
Recall that the depreciation rate $\mu>0$ and there exists $\rho'\in[0,\rho)$ such that $\sup_{t\geq0}e^{-\frac{\rho' t}{2}}|m_t^P|<\infty$. This yields that
\begin{align*}
  \sup_{t\geq0}e^{-\frac{\rho' t}{2}}\left|\int_{t}^{\infty}e^{-(\mu+\rho)(s-t)}m^P_{s}ds\right|
  &\leq C\sup_{t\geq0}e^{\left(\mu+\rho-\frac{\rho'}{2}\right)t}\int_{t}^{\infty}e^{-(\mu+\rho-\frac{\rho' }{2})s}ds\\
  &=\frac{C}{\mu+\rho-\frac{\rho' }{2}}<\infty,
\end{align*}
and for all $\tilde{\rho}\in[0,\rho)$,
\begin{align*}
  \sup_{t\geq0}|\widetilde{D_2}|e^{-\frac{\tilde{\rho} t}{2}}e^{(\mu+\rho)t}\geq\sup_{t\geq0}|\widetilde{D_2}|e^{\left(\mu+\frac{\rho}{2}\right)t}=+\infty.
\end{align*}
This contradicts with $g\in C_{\rho/2}([0,\infty);\R)$. Similarly, we can prove that the 2nd equation of \eqref{eq:ansatz-gh} has a unique solution in $C_\rho([0,\infty);\R)\cap C^1([0,\infty);\R)$. Thus, we complete the proof of the lemma.
\end{proof}

With the help of Lemma \ref{lem:wellpo-HJB}, we can establish the optimal solution to the control problem \eqref{eq:repre-cont-pr} as follows:
\begin{proposition}\label{prop:repre-optimal}
For any fixed $m^X\in C_{\rho/2}([0,\infty);\R)$, let $(g,h)$ be the unique solution to Eq.~\eqref{eq:ansatz-gh} given in Lemma~\ref{lem:wellpo-HJB}. Then, as an element of $\mathbb{U}$, $u_t^*:=\frac{1}{2r}g_t$, $t\geq0$
is the unique optimal strategy to \eqref{eq:repre-cont-pr}. Moreover, the optimal reward functional is given by
\begin{align}\label{eq:optimalreward}
    \sup_{u\in\mathbb{U}}R(u)=V(0,x_0)=g_0x_0+h_0.
\end{align}
\end{proposition}

\begin{proof}
For any $u\in\mathbb{U}$, it follows from \eqref{eq:repre-cont-pr} and It\^{o}'s rule that
\begin{align*}
 e^{-\rho t}(g_tX_t^u+h_t)&=g_0x_0+h_0
 +\int_{0}^{t}e^{-\rho  s}\left[-(1-c)m^P_sX_s^u+u_sg_s-\frac{1}{4r}g_s^2\right]ds\\
  &\quad+\sigma\int_{0}^{t}e^{-\rho s}g_{s}X_{s}^udW_{s}.
\end{align*}
Then, we have from \eqref{eq:repre-cont-pr} and \eqref{eq:Vtx} that, for all $u\in\mathbb{U}$,
\begin{align*}
R(u)&=g_0x_0+h_0-\lim_{t\to\infty}\mathbb{E}[e^{-\rho t}(g_tX_t^u+h_t)]
+\lim_{t\to\infty}\int_{0}^{t}e^{-\rho s}\Ex\left[-ru_s^2+u_sg_s-\frac{1}{4r}g_s^2\right]ds\\
&=g_0x_0+h_0-r\int_{0}^{\infty}e^{-\rho s}\Ex\left[\left(u_s-\frac{1}{2r}g_s\right)^2\right]ds\\
&\leq g_0x_0+h_0=V(0,x_0).
\end{align*}
For 2nd equality in the above display, we used Lemma \ref{lem:estimates} and the fact that $g\in C_{\rho/2}([0,\infty);\R)$ and $h=(h_t)_{t\geq0}\in C_{\rho}([0,\infty);\R)$. In terms of $u^*=\frac{1}{2r}g_t$, we have $R(u^*)=V(0,x_0)$. Moreover, if $u\in\mathbb{U}$ and $\Px(u_t\neq u^*_t)>0$ for some $t\geq0$, by the right continuity of $t\mapsto u_t$ and $t\mapsto g_t$, $\int_{0}^{\infty}e^{-\rho s}\Ex[(u_s-\frac{1}{2r}g_s)^2]ds>0$, i.e., $R(u)<V(0,x_0)$. This yields that $u^*=(u_t^*)_{t\geq0}$ is the unique optimal control to \eqref{eq:repre-cont-pr}, and hence $V(0,x_0)=g_0x_0+h_0=\sup_{u\in\mathbb{U}}R(u)$.
Thus, we complete the proof of the proposition. 
\end{proof}

\subsection{The fixed point problem}

For a given $m^X\in C_{\rho/2}([0,\infty);\R)$ with $m_0^X=x_0\in\R$, the associated {\it consistence condition} in the MFG is given by
\begin{equation}\label{eq:consis}
  m^X_t=\Ex[X_t^{u^*}],\quad t\geq0,
\end{equation}
where $X^u=(X_t^u)_{t\geq0}$ and $u^*=(u^*_t)_{t\geq0}\in\mathbb{U}$ are respectively given by \eqref{eq:repre-cont-pr} and in Proposition \ref{prop:repre-optimal}. From \eqref{eq:repre-cont-pr} and Proposition \ref{prop:repre-optimal}, it follows that
\begin{align}\label{eq:Extustar}
   \Ex[X_t^{u^*}]
  & =x_0 - \mu\int_0^t\Ex[X_s^{u^*}]ds+\int_0^t\frac{1}{2r}g_sds,
\end{align}
where $g$ is given by \eqref{eq:gh}. By solving Eq.~\eqref{eq:Extustar} with unknown $\Ex[X_t^{u^*}]$, we arrive at
\begin{align}\label{eq:def-mapping00}
 &\Ex[X^{u^*}_t]=e^{-\mu t}x_0+\frac{1-c}{2r}\int_{0}^{t}e^{-\mu(t-s)}\left(\int_{s}^{\infty}e^{-(\mu+\rho)(v-s)}m^P_{v}dv\right) ds.
\end{align}
where $m^P=(m_t^P)_{t\geq0}$ satisfies \eqref{eq:lim-m^P}. 
By Lemma~\ref{lem:estimates} with $g\in C_{\rho/2}([0,\infty),\R)$, $(\Ex[X^{u^*}_t])_{t\geq0}\in C_{\rho/2}([0,\infty),\R)$. Thus, we define $\mathcal{L}(m^X):C_{\rho/2}([0,\infty),\R)\to C_{\rho/2}([0,\infty),\R)$ as:
\begin{align}\label{eq:amppL}
[\mathcal{L}(m^X)]_t:=\Ex[X^{u^*}_t],\quad \forall t\geq0.
\end{align}
Then, the consistence condition \eqref{eq:consis} is equivalent to the existence of fixed points of $\mathcal{L}$ in $C_{\rho/2}([0,\infty);\R)$. That is, we want to find a fixed point $\bar{m}^X\in C_{\rho/2}([0,\infty);\R)$ such that $\bar{m}^X=\mathcal{L}(\bar{m}^X)$.

Equivalently, the fixed point $\bar{m}^X$ (if it exists) together with the price process $\bar{m}^P$ and the optimal control $\bar{u}^*$, should satisfy
{\small\begin{align}\label{eq:pxu}
\begin{cases}
  \displaystyle  d\bar{m}^P_t=\alpha[\beta-(1-\lambda\gamma)\bar{m}^X_t-\bar{m}^P_t]dt,~ \bar{m}^P_0=p_0,\\[0.8em]
  \displaystyle d\bar{m}^X_t=(-\mu \bar{m}^X_t+\bar{u}^*_t)dt,~ \bar{m}^X_0=x_0,\\[0.8em]
  \displaystyle d\bar{u}^*_t=\left[(\mu+\rho)\bar{u}^*_t-\frac{1-c}{2r}\bar{m}^P_t\right]dt,~ \bar{u}^*_0=\frac{1-c}{2r}\int_{0}^{\infty}e^{-(\mu+\rho)t}\bar{m}^P_tdt.
\end{cases}
\end{align}}
Thus, the existence and uniqueness of a fixed point is equivalent to the well-posedness of \eqref{eq:pxu}.

We next consider the following cubic equation with unknown single variable $K$ given by
\begin{align}\label{eq:charact}
K^3+(\alpha-\rho)K^2-AK-B=0,
\end{align}
where the coefficients of \eqref{eq:charact} are given by
\begin{align}\label{eq:AB}
A:=\mu^2+\rho\mu+\alpha\rho>0,\quad
B:= \alpha\left[\mu(\mu+\rho)+\frac{(1-\lambda\gamma)(1-c)}{2r}\right]>0.
\end{align}
The discriminant of the cubic equation \eqref{eq:charact} is given by
\begin{align}\label{eq:deltadis}
\Delta:=-27B^2+[18(\alpha-\rho)A+4(\alpha-\rho)^3]B+[(\alpha-\rho)^2A^2+4A^3].
\end{align}
We further have that
\begin{itemize}
    \item if $\Delta>0$, Eq.~\eqref{eq:charact} has three roots $K_1,K_2,K_3$ satisfying $K_1<K_2<0<K_3$;
    \item if $\Delta=0$, Eq.~\eqref{eq:charact} has two roots $K_1,K_2$ satisfying $K_1<0<K_2$;
    \item if $\Delta<0$, Eq.~\eqref{eq:charact} has three roots $K_1,K_2\in\mathbb{C}$ and $K_3>0$, where $K_1=K_R+K_I i$ and $K_2=K_R-K_I i$ for some $K_R<0$ and  $K_I\neq0$.
\end{itemize}

Let $C_b([0,\infty);\R)$ be the set of bounded continuous functions from $[0,\infty)$ to $\R$. The main result of this section is given as follows:
\begin{theorem}\label{thm:fixed-point}
Eq.~\eqref{eq:pxu} has a unique solution $(\bar{m}^P,\bar{m}^X,\bar{u}^*)\in C_b([0,\infty);\R)^3$.  Moreover, the solution component $\bar{m}^P=(\bar{m}_t^P)_{t\geq0}$ admits the following closed-form representation:\\

\noindent{\bf(i)} if $\Delta>0$,
 \begin{align*}
      \bar{m}^P_t&=\frac{-\alpha(\beta-(1-\lambda\gamma)x_0-p_0)+(p_0-p^*)K_2}{K_2-K_1}e^{K_1 t}\nonumber\\
      &\quad+\frac{\alpha(\beta-(1-\lambda\gamma)x_0-p_0)-(p_0-p^*)K_1}{K_2-K_1}e^{K_2 t}+p^*.
  \end{align*}

\noindent{\bf(ii)} if $\Delta=0$,
\begin{align*}
       \bar{m}^P_t&=p^*+e^{K_1t}\{p_0-p^*+[\alpha(\beta-(1-\lambda\gamma)x_0-p_0)-K_1(p_0-p^*)]t\}.
      \end{align*}

\noindent{\bf(iii)} if $\Delta<0$,
 \begin{align*}
       \bar{m}^P_t=p^*+e^{K_R t}\bigg[&\frac{-\alpha(\beta-(1-\lambda\gamma)x_0-p_0)+(p_0-p^*)K_I}{K_I-K_R}\cos(K_I t)\\
        &\quad+\frac{\alpha(\beta-(1-\lambda\gamma)x_0-p_0)-(p_0-p^*)K_R}{K_I-K_R}\sin(K_R t)\bigg].
      \end{align*}
The parameter in the above expressions is defined by
\begin{align}\label{eq:stable-p}
p^*=\frac{\alpha\mu\beta(\mu+\rho)}{B}.
\end{align}
\end{theorem}

\begin{remark}Since $\bar{m}^X=(\bar{m}_t^X)_{t\geq0}$ is a fixed point of the mapping ${\cal L}$ defined by \eqref{eq:amppL}, we may explain $\bar{m}^X$ as the expectation of production output rate under the optimal output adjusted rate $\bar{u}^*$ of the representative firm. Then, $\bar{m}^P=(\bar{m}_t^P)_{t\geq0}$ may be explained as the price process of
goods produced by the representative firm under the optimal output adjusted rate. Moreover, by Theorem \ref{thm:fixed-point}, we have
\begin{align*}
  (\bar{m}^P_t,\bar{m}^X_t,\bar{u}^*_t)\to \left(p^*,\frac{\beta-p^*}{1-\lambda\gamma},\frac{\mu(\beta-p^*)}{1-\lambda\gamma}\right),~t\to\infty
\end{align*}
Together with Lemma~\ref{lem:cov-L2P+mean} below, $p^*$ given by \eqref{eq:stable-p} may be explained as the stationary equilibrium price in
the commodity market (in the sense that the number of firms grows large).
\end{remark}

\begin{proof}[Proof of Theorem \ref{thm:fixed-point}]
By differentiating on both sides of the 1st Eq.~in \eqref{eq:pxu}, we have
\begin{align*}
\frac{1}{\alpha}\frac{d^2\bar{m}^P_t}{dt^2}+\frac{d\bar{m}^P_t}{dt}+(1-\lambda\gamma)\frac{d\bar{m}^X_t}{dt}=0.
\end{align*}
It then follows from the 2nd equation in \eqref{eq:pxu}  that
\begin{align*}
\frac{1}{\alpha}\frac{d^2\bar{m}^P_t}{dt^2}+\frac{d\bar{m}^P_t}{dt}-(1-\lambda\gamma)\mu \bar{m}^X_t+(1-\lambda\gamma)\bar{u}^*_t=0.
\end{align*}
Substitute the 1st Eq. in \eqref{eq:pxu} into the above display, we have
\begin{equation}\label{eq:app-p1}
  \frac{1}{\alpha}\frac{d^2\bar{m}^P_t}{dt^2}+\left(1+\frac{\mu}{\alpha}\right)\frac{d\bar{m}^P_t}{dt}+\mu \bar{m}^P_t+(1-\lambda\gamma)\bar{u}^*_t-\mu\beta=0.
\end{equation}
By differentiating on both sides of \eqref{eq:app-p1}, it holds that
{\small\begin{align*}
\frac{1}{\alpha}\frac{d^3\bar{m}^P_t}{dt^3}+\left(1+\frac{\mu}{\alpha}\right)\frac{d^2\bar{m}^P_t}{dt^2}+\mu \frac{d\bar{m}^P_t}{dt}+(1-\lambda\gamma)\frac{d\bar{u}^*_t}{dt}=0.
\end{align*}}
Using the 3rd Eq. of \eqref{eq:pxu}, it follows that
\begin{align*}
\frac{1}{\alpha}\frac{d^3\bar{m}^P_t}{dt^3}+\left(1+\frac{\mu}{\alpha}\right)\frac{d^2\bar{m}^P_t}{dt^2}+\mu \frac{d\bar{m}^P_t}{dt}-\frac{(1-\lambda\gamma)(1-c)}{2r}\bar{m}^P_t+(\mu+\rho)(1-\lambda\gamma)\bar{u}^*_t=0.
\end{align*}
Thus, using \eqref{eq:app-p1}, we deduce that
{\small\begin{equation*}
    u^*_t=\frac{1}{1-\lambda\gamma}\left[\mu\beta-\left(\frac{1}{\alpha}\frac{d^2\bar{m}^P_t}{dt^2}+\left(1+\frac{\mu}{\alpha}\right)\frac{d\bar{m}^P_t}{dt}+\mu \bar{m}^P_t\right)\right].
\end{equation*}}
Substitute it into the above equation, $\bar{m}^P$ obeys that
\begin{equation}\label{eq:p-diff}
  \frac{d^3\bar{m}^P_t}{dt^3}+(\alpha-\rho)\frac{d^2\bar{m}^P_t}{dt^2}-A \frac{d\bar{m}^P_t}{dt}-B\bar{m}^P_t +\alpha\mu\beta(\mu+\rho)=0,
\end{equation}
where $\bar{m}^P_0=p_0$, the constants $A,B$ are defined by \eqref{eq:AB}, and $\lim_{t\to0}\frac{d\bar{m}^P_t}{dt}=\alpha[\beta-(1-\lambda\gamma)x_0-p_0]$.
Note that $p^*=\alpha\mu\beta(\mu+\rho)/B>0$ is a special solution of \eqref{eq:p-diff}. Then, the characteristic equation corresponding to the homogeneous part of \eqref{eq:p-diff} is given by
\begin{align*}
f(K):=K^3+(\alpha-\rho)K^2-AK-B=0.
\end{align*}
The two real roots of $f'(K)=3K^2+2(\alpha-\rho)-A=0$ satisfy
\begin{align*}
\frac{-(\alpha-\rho)-\sqrt{(\alpha-\rho)^2+3A}}{3}<0, \quad \frac{-(\alpha-\rho)+\sqrt{(\alpha-\rho)^2+3A}}{3}>0.
\end{align*}
Then, by $f(0)=-B<0$, the characteristic equation \eqref{eq:charact}  has actually one positive root denoted by $K_3$. Using the boundness of $\bar{m}^P=(\bar{m}_t^P)_{t\geq0}$, we can only use the other two roots of the characteristic equation, as the positive root yields a term with form $e^{K_3 t}$. Hence, the solution of \eqref{eq:p-diff} stated in Theorem \ref{thm:fixed-point} is uniquely determined via the discriminant $\Delta=-27B^2+[18(\alpha-\rho)A+4(\alpha-\rho)^3]B+[(\alpha-\rho)^2A^2+4A^3]$ of $f(K)=0$. Thus, we complete the proof of the theorem. 
\end{proof}

\section{Approximating Nash Equilibrium}\label{sec:approx-Nash}

This section will establish an approximating Nash equilibrium of the MFG problem described in Section~\ref{sec:MFG}.

Let $\bar{m}^X=(\bar{m}^X_t)_{t\geq0}$ and $\bar{m}^P=(\bar{m}^P)_{t\geq0}$ be the solution components of \eqref{eq:pxu} (c.f. Theorem \ref{thm:fixed-point}). We introduce $\bar{g}^i=(\bar{g}_t^i)_{t\geq0}\in C_{\rho/2}([0,\infty);\R)\cap C^1([0,\infty);\R)$ satisfying
\begin{equation}\label{eq:bar-gi}
  \frac{d\bar{g}_t^i}{dt}=(\rho+\mu_i)\bar{g}_t^i-(1-c_i)\bar{m}_t^P.
\end{equation}
This is equivalent to
\begin{equation}\label{eq:bar-gi-ex}
  \bar{g}_t^i=(1-c_i)\int_{t}^{\infty}e^{-(\mu_i+\rho)(s-t)}\bar{m}^P_{s}ds,\quad t\geq0.
\end{equation}
Based upon \eqref{eq:bar-gi-ex}, let us define
\begin{equation}\label{eq:u*i}
  u^{*,i}_t=\frac{1}{2r_i}\bar{g}_t^i,\quad t\geq0.
\end{equation}
We next rewrite the reward functional \eqref{eq:reward-func} for firm $i$, but highlight the dependence on the number $n$ of firms, i.e.,
{\small\begin{equation}\label{eq:Rni}
  R_i^{(n)}(u)=\Ex\left[\int_{0}^{\infty}e^{-\rho t}\left((1-c_i)P_t^{u,(n)} X_t^{i,u^i}-r_i(u_t^i)^2\right)dt\right],
\end{equation}}
where $P^{u,{(n)}}=(P^{u,{(n)}}_t)_{t\geq0}$ is the price process \eqref{eq:Price} with the dependence on $n$.

The main result of this section is stated as follows:
\begin{theorem}\label{thm:main}
Let assumptions {\Asa}-{\Asc} hold. Recall $u^{*,(n)}:=(u^{*,1},\ldots,u^{*,n})$ defined by \eqref{eq:u*i}, we have, for all $i=1,\ldots,n$,
\begin{align}\label{eq:approNElimit}
    \sup_{u^i\in\mathbb{U}}R_i^{(n)}(u^i,u^{*,-i})\leq R_i^{(n)}(u^{*,(n)})+\varepsilon_n
\end{align}
with $\lim_{n\to\infty}\varepsilon_n=0$ and the policy $(u^i,u^{*,-i})$ being defined in Definition \ref{def:Nash}.
\end{theorem}

To prove Theorem \ref{thm:main}, we need the following auxiliary results. Consider the $i$-th firm's state process  $X^{i,*}=(X^{i,*}_t)_{t\geq0}$ with control $u^{*,i}$ given by: $X_0^{i,*}=X_0^i$,
\begin{equation}\label{eq:X*i}
  dX_t^{i,*}=X_t^{i,*}(-\mu_idt+\sigma_i dW_t^i)+u^{*,i}_tdt.
\end{equation}
The corresponding price dynamics of goods is then given by: $P_0^{*,(n)}=p_0$,
\begin{align}\label{eq:Pn*}
  dP_t^{*,(n)}=\alpha\left(\beta-\frac{1}{n}\sum_{i=1}^n X_t^{i,*}-P_t^{*,(n)}\right)dt+\frac{\alpha }{n}\sum_{i=1}^{n}\gamma_iX_{t}^{i,*}dN_t^{i}.
\end{align}
Then, we have that
\begin{lemma}\label{lem:es-L2P+mean}
Let assumptions {\Asa}-{\Asc} hold. Then, there exists a constant $D_2>0$ independent of $n$ such that
\begin{equation}\label{eq:L2P+mean}
  \sup_{n\geq1}\sup_{t\geq0}\Ex\left[\left|P_t^{*,(n)}\right|^2+\left|\overline{X}_t^{*,(n)}\right|^2\right]\leq D_2,
\end{equation}
where $\overline{X}_t^{*,(n)}$ for $t\geq0$ is defined by
\begin{align}\label{eq:aveXstar}
    \overline{X}_t^{*,(n)}:=\frac{1}{n}\sum_{i=1}^n X_t^{*,i},\quad t\geq0.
\end{align}
\end{lemma}

\begin{proof}
Note that $\bar{m}^P\in C_b([0,\infty);\R)$. Then, by {\Asb}, $\bar{g}^i$ defined by \eqref{eq:bar-gi-ex} is uniformly bounded, i.e.,
\begin{align}\label{eq:bg-bound}
  \sup_{i\geq1}\|\bar{g}^i\|_{\infty}:=\sup_{i\geq1}\sup_{t\geq0}|\bar{g}_t^i|<\infty.
\end{align}
In addition, the process $X^{i,*}=(X^{i,*}_t)_{t\geq0}$ defined in \eqref{eq:X*i} has the following closed-form representation:
\begin{align*}
X_t^{i,*}=e^{-\left(\mu_i+\frac{\sigma_i^2}{2}\right)t}e^{\sigma_iW_t^i}X_0^i+\frac{1}{2r_i} \int_{0}^{t}\bar{g}_s^ie^{-\left(\mu_i+\frac{\sigma_i^2}{2}\right)(t-s)}e^{\sigma_i(W_t^i-W_s^i)}ds.
\end{align*}
It follows from assumptions {\Asa} and {\Asb} that
\begin{align}\label{eq:X*-L2bound}
  \sup_{i\geq1}\sup_{t\geq0} \Ex\left[\left|X_t^{i,*}\right|^2\right]
  &\leq \sup_{i\geq1}\sup_{t\geq0}C\left(e^{-(2\mu_i-\sigma_i^2)t}+\int_0^t e^{-(2\mu_i-\sigma_i^2)(t-s)}ds\right)\nonumber\\
  &\leq\sup_{i\geq1}\sup_{t\geq0}C\left(1+\frac{1}{(2\mu_i-\sigma_i^2)}(1-e^{-(2\mu_i-\sigma_i^2)t})\right)<\infty,
\end{align}
where $C>0$ is a generic positive constant independent of $i$ and $t$. 
In light of Jensen's inequality with \eqref{eq:X*-L2bound}, it holds that
\begin{align*}
\sup_{n\geq1}\sup_{t\geq0}\Ex[|\overline{X}_t^{*,(n)}|^2] \leq\sup_{n\geq1}\frac{1}{n}\sum_{i=1}^n\sup_{t\geq0}\Ex[|X_t^{i,*}|^2]<\infty.
\end{align*}
Using Theorem V.56 of \cite{Protter05}, the price process $P^{*,(n)}=(P_t^{*,(n)})_{t\geq0}$ admits that
\begin{align*}
    P_t^{*,(n)}=e^{-\alpha t}p_0+\alpha\int_0^t e^{-\alpha(t-s)}(\beta-\overline{X}_s^{*,(n)})ds+\frac{\alpha}{n}\sum_{i=1}^{n}\gamma_i\int_{0}^{t}e^{-\alpha(t-s)}X_s^{i,*}dN_s^i.
\end{align*}
By the independence of $N^i$ for $i\geq1$, and {\Asb}, there exists $C>0$ independent of $n,t$ s.t.
\begin{align*}
  \sup_{n\geq1}\sup_{t\geq0} \Ex\left[\left|P_t^{*,(n)}\right|^2\right]
  &\leq \sup_{n\geq1}\sup_{t\geq0}C\Bigg\{1+\int_0^t e^{-2\alpha(t-s)}ds+\frac{1}{n}\sum_{i=1}^n\lambda_i\int_0^t e^{-2\alpha(t-s)}\Ex[|X_s^{i,*}|^2]ds\Bigg\}\\
  &\leq \sup_{n\geq1}\sup_{t\geq0}C\Bigg[1+\int_0^t e^{-2\alpha(t-s)}ds\Bigg]=\sup_{n\geq1}\sup_{t\geq0}C\Bigg[1+\frac{1}{2\alpha}\left(1-e^{-2\alpha t}\right)\Bigg]<\infty.
\end{align*}
This completes the proof of the lemma.
\end{proof}

The next lemma gives  $(P_t^{*,(n)},\overline{X}_t^{*,(n)})\to (\bar{m}_t^P,\bar{m}_t^X)$ as $n\to\infty$ in $L^2$-sense.
\begin{lemma}\label{lem:cov-L2P+mean}
Let assumptions {\Asa}-{\Asc} hold. Recall that $(\bar{m}^X,\bar{m}^P)\in C_b([0,\infty);\R)^2$ is the fixed point obtained in Theorem \ref{thm:fixed-point}. Then
{\small\begin{align}\label{eq:convergence000}
    \lim_{n\to\infty}\sup_{t\geq0}\Ex\left[\left|P_t^{*,(n)}-\bar{m}^P_t\right|^2 +\left|\overline{X}_t^{*,(n)}-\bar{m}^X_t\right|^2\right]=0.
\end{align}}
\end{lemma}

\begin{proof}
Let $m^i_t:=\Ex[X^{i,*}_t]$. Then, for all $t\geq0$,
{\small\begin{align*}
  \Ex\left[\left|\overline{X}_t^{*,(n)}-\bar{m}^X_t\right|^2\right]
 & = \Ex\left[\left|\frac{1}{n}\sum_{i=1}^{n}(X_t^{i,*}-m_t^i) +\frac{1}{n}\sum_{i=1}^{n}(m_t^i-\bar{m}^X_t)\right|^2\right]\nonumber \\
  &\leq 2\Ex\left[\left|\frac{1}{n}\sum_{i=1}^{n}(X_t^{i,*}-m_t^i)\right|^2\right] +2\left|\frac{1}{n}\sum_{i=1}^{n}(m_t^i-\bar{m}^X_t)\right|^2.
\end{align*}}
Note that $m^i=(m^i_t)_{t\geq0}$ solves that
\begin{equation}\label{eq:mi}
  dm_t^i=-\mu_im_t^idt+\frac{1}{2r_i}\bar{g}_t^idt,\quad m_0^i=\Ex[X_0^i].
\end{equation}
Then, by \eqref{eq:bg-bound} and the assumption {\Asb}, we have
\begin{align*}
    \sup_{i\geq1}\|m^i\|_\infty<\infty.
\end{align*}
Thanks to \eqref{eq:X*-L2bound}, we deduce that
\begin{align*}
    \sup_{i\geq 1}\sup_{t\geq0}\Ex\left[\left|X^{i,*}_t-m^i_t\right|^2\right]<\infty.
\end{align*}
Using the independence of $(X^{i,*})_{i=1}^n$, it follows that
\begin{align*}
\lim_{n\to\infty}\Ex\left[\left|\frac{1}{n}\sum_{i=1}^{n}(X_t^{i,*}-m_t^i)\right|^2\right]=0,\quad \forall t\geq0.
\end{align*}
By Jensen's inequality and the inequality $a^2-b^2\leq 2(a\vee b)|a-b|$, we get
{\small\begin{align*}
  &\sup_{|t-s|<\delta}\left|\Ex\left[\left|\frac{1}{n}\sum_{i=1}^{n}(X_t^{i,*}-m_t^i)\right|^2\right] -\Ex\left[\left|\frac{1}{n}\sum_{i=1}^{n}(X_s^{i,*}-m_s^i)\right|^2\right]\right|\\
  &\quad\leq \frac{C}{n}\sum_{i=1}^{n}\sup_{|t-s|<\delta} \left\{\Ex[|X_t^{i,*}-X_s^{i,*}|]+|m_t^i-m_s^i|\right\}\leq C\delta,
\end{align*}}
where $C>0$ is a generic positive constant independent of $n$. Thus, we arrive at
\begin{align*}
\lim_{n\to\infty}\sup_{t\geq0}\Ex\left[\left|\frac{1}{n}\sum_{i=1}^{n}(X_t^{i,*}-m_t^i)\right|^2\right]=0.
\end{align*}
Similarly, it can be deduced that
\begin{equation}\label{eq:meanCov-p1}
  \lim_{n\to\infty}\sup_{t\geq0}\left|\frac{1}{n}\sum_{i=1}^{n}(m_t^i-\bar{m}^X_t)\right|^2=0.
\end{equation}
These estimates conclude that
\begin{equation}\label{eq:meanCov}
  \lim_{n\to\infty}\sup_{t\geq0}\Ex\left[\left|\overline{X}_t^{*,(n)}-\bar{m}^X_t\right|^2\right]=0.
\end{equation}

We next show that $P_t^{*,(n)}$ converges to $\bar{m}^P_t$ in $L^2$ uniformly in $t$ as $n\to\infty$. We deduce from \eqref{eq:pxu} and \eqref{eq:Pn*} that
\begin{align*}
  &  P_t^{*,(n)}-\bar{m}^P_t=\alpha\int_{0}^{t}e^{-\alpha(t-s)}(\bar{m}^X_s-\overline{X}_s^{*,(n)})ds+\frac{\alpha}{n}\sum_{i=1}^{n}\gamma_i\int_{0}^{t}e^{-\alpha(t-s)}X_{s-}^{i,*}dN_s^i-\alpha\lambda\gamma\int_{0}^{t} e^{-\alpha(t-s)}\bar{m}^X_sds.
\end{align*}
Therefore, it holds that
\begin{align*}
  \Ex\left[\left|P_t^{*,(n)}-\bar{m}^P_t\right|^2\right]
  &\leq C\Bigg\{ \int_{0}^{t}e^{-2\alpha(t-s)}\Ex[|\bar{m}^X_s-\overline{X}_s^{*,(n)}|^2]ds\\
  &\quad+\Ex\left[\left|\frac{1}{n}\sum_{i=1}^{n}\gamma_i\left(\int_{0}^{t}e^{-\alpha(t-s)}X_{s}^{i,*}dN_s^i -\lambda_i\int_{0}^{t}e^{-\alpha(t-s)}m_s^ids\right)\right|^2\right]\\
  &\quad+\left|\frac{1}{n}\sum_{i=1}^{n}\lambda_i\gamma_i\int_{0}^{t}e^{-\alpha(t-s)}m_s^ids -\alpha\lambda\gamma\int_{0}^{t} e^{-\alpha(t-s)}\bar{m}^X_sds\right|^2 \Bigg\}.
\end{align*}
Using \eqref{eq:meanCov}, the 1st term of RHS of the above display converges to $0$ uniformly in $t$ as $n\to\infty$. Moreover, by the independence of $(N^i)_{i\geq 1}$, it is a direct result from the Doob's maximal inequality that the 2nd term also converges to $0$ uniformly in $t$, as $n\to\infty$. Thus, using \eqref{eq:meanCov-p1} and the assumption {\Asb}, it follows that
\begin{align*}
\lim_{n\to\infty}\sup_{t\geq0}\Ex\left[\left|P_t^{*,(n)}-\bar{m}^P_t\right|^2\right]=0.
\end{align*}
Then, the desired result follows from  \eqref{eq:meanCov}.
\end{proof}

\begin{lemma}\label{lem:bound-u}
Let assumptions {\Asa}-{\Asc} hold. Define the control set as follows:
\begin{equation}\label{eq:rational-choice}
  \widetilde{\mathbb{U}}^i=\{u^i\in\mathbb{U};~R_i^{(n)}(u^i,u^{*,-i})\geq0\}.
\end{equation}
For the the decentralized strategy $u^{*,i}$ for $i=1,\ldots,n$ defined by \eqref{eq:u*i}, there exist positive constants $n_0$ and $D_3$ such that $\sup_{u^i\in\widetilde{\mathbb{U}}^i}\|u^i\|_\rho^2\leq D_3$ whenever $n\geq n_0$.
\end{lemma}

\begin{proof}
Let $\widehat{P}^{u^i,(n)}=(\widehat{P}_t^{u^i,(n)})_{t\geq0}$ be the  price process  with the policy $(u^i,u^{*,-i})$. The limiting reward functional for firm $i$ is defined by
{\small\begin{equation}\label{eq:barRi}
  \bar{R}_i(u^i):=\Ex\left[\int_{0}^{\infty}e^{-\rho t}\left((1-c_i)\bar{m}^P_t X_t^{i,u^i}-r_i(u_t^i)^2\right)dt\right].
\end{equation}}
By H\"{o}lder's inequality, we have that
{\small\begin{align}\label{eq:lem-B-total}
  R_i^{(n)}(u^i,u^{*,-i})
  &=(1-c_i)\Ex\left[\int_0^\infty e^{-\rho t}(\widehat{P}^{u^i,(n)}_t-\bar{m}^P_t)X_t^{i,u^i}dt\right]+\bar{R}_i(u^i)\nonumber\\
  &\leq\left\{\Ex\left[\int_0^\infty e^{-\rho t}|\widehat{P}^{u^i,(n)}_t-\bar{m}^P_t|^2dt\right]\|X^{i,u^i}\|_\rho^2\right\}^{1/2}+\bar{R}_i(u^i)\nonumber\\
 &=: I+\bar{R}_i(u^i).
\end{align}}
Note that
\begin{align*}
\Ex\left[\int_0^\infty e^{-\rho t}|\widehat{P}^{u^i,(n)}_t-\bar{m}^P_t|^2dt\right]
\leq 2\Ex\left[\int_0^\infty e^{-\rho t}|\widehat{P}^{u^i,(n)}_t-P_t^{*,(n)}|^2dt\right]+2\Ex\left[\int_0^\infty e^{-\rho t}|P_t^{*,(n)}-\bar{m}^P_t|^2dt\right],
\end{align*}
where $P^{*,(n)}=(P_t^{*,(n)})_{t\geq0}$ is the price process with the policy $u^{*,(n)}:=(u^{*,1},\ldots,u^{*,n})$ which is defined by \eqref{eq:Pn*}. Thanks to Lemma \ref{lem:cov-L2P+mean}, the 2nd term on RHS of the above display converges to $0$ as $n\to\infty$. On the other hand, using Theorem V.56 of \cite{Protter05}, it follows that
\begin{align*}
\widehat{P}^{u^i,(n)}_t-P_t^{*,(n)}&=\frac{\alpha}{n}\bigg[\int_0^t e^{-\alpha(t-s)}(X_s^{i,*}-X_s^{i,u^i})ds+\gamma_i\int_0^t e^{-\alpha(t-s)}(X_s^{i,*}-X_s^{i,u^i})dN_s^i\bigg].
\end{align*}
By Lemma \ref{lem:estimates} with {\Asb}, we have $\sup_{i\geq1}\|X^{i,*}\|_\rho<\infty$. Then, by Fubini's theorem, it holds that
\begin{align*}
  \Ex\left[\int_0^\infty e^{-\rho t}\left|\widehat{P}^{u^i,(n)}_t-P_t^{*,(n)}\right|^2dt\right]
  &\leq \frac{C}{n^2}\int_0^\infty e^{-\rho s}\Ex[|X_s^{i,*}-X_s^{i,u^i}|^2]\left(\int_s^\infty e^{-2\alpha(t-s)}dt\right)ds\\
  &\leq \frac{C}{n^2}(\|X^{i,*}\|_\rho^2+\|X^{i,u^i}\|_\rho^2)\leq\frac{C}{n^2}(1+\|X^{i,u^i}\|_\rho^2),
\end{align*}
where $C>0$ is a positive constant independent of $i,n$ and the choice of $u^i$. Thus, Lemma \ref{lem:estimates} yields that, there exists a constant $\widetilde{C}_1>0$ independent of the choice of $u^i$ s.t.
\begin{equation}\label{eq:lem-B-p1}
  I\leq \widetilde{C}_1+\frac{\widetilde{C}_1}{n}\|u^i\|_\rho^2.
\end{equation}
Since $\|\bar{m}^P\|_{\infty}<\infty$, we have, for any $\delta>0$,
\begin{align*}
  \bar{R}_i(u^i)
  &\leq \|\bar{m}^P\|_{\infty}\Ex\left[\int_0^\infty e^{-\rho t}X_t^{i,u^i}dt\right]-r_i\|u^i\|_\rho^2\leq \|\bar{m}^P\|_{\infty}\Ex\left[\int_0^\infty e^{-\rho t}\left(\delta(X_t^{i,u^i})^2+\frac{1}{4\delta}\right)dt\right]-r_i\|u^i\|_\rho^2\\
  &\leq \|\bar{m}^P\|_{\infty}\left(\frac{1}{4\delta\rho}+\delta D_1\right) -\left(r_i-\delta\|\bar{m}^P\|_{\infty}D_1\right)\|u^i\|_\rho^2,
\end{align*}
where $D_1$ is the constant given in Lemma \ref{lem:estimates}. Then, there exists $\delta>0$ small enough s.t. $r_i-\delta\|\bar{m}^P\|_{\infty}D_1=:\eta>0$. This implies that
\begin{equation}\label{eq:lem-B-p2}
  \bar{R}_i(u^i)\leq \widetilde{C}_2-\eta\|u^i\|_\rho^2,
\end{equation}
where $\widetilde{C}_2$ is a constant depending on $\eta$, but it is independent of the choice of $u^i$. Plugging \eqref{eq:lem-B-p1} and \eqref{eq:lem-B-p2} into \eqref{eq:lem-B-total}, we have
\begin{align*}
0\leq R_i^n(u^i,u^{*,-i})\leq\widetilde{C}_1+\widetilde{C}_2 -\left(\eta-\frac{\widetilde{C}_1}{n}\right)\|u^i\|_\rho^2,
\end{align*}
for all $u^i\in\widetilde{\mathbb{U}}^i$. Let $n_0=\inf\{n\geq 1|n>\widetilde{C}_1/\eta\}$. Then, for any $n\geq n_0$,
\begin{align*}
\|u^i\|_\rho^2\leq \frac{\widetilde{C}_1+\widetilde{C}_2}{\eta-(\widetilde{C}_1/n_0)}=:D_3,
\end{align*}
where $D_3$ is independent of $i$ and the choice of $u^i$.
\end{proof}


\begin{proof}[Proof of Theorem \ref{thm:main}]
By the definition \eqref{eq:rational-choice}, it suffices to prove that
\begin{equation}\label{eq:ine0000}
  \sup_{u^i\in\widetilde{\mathbb{U}}^i}R_i^{(n)}(u^i,u^{*,-i})\leq R_i^{(n)}(u^{*,(n)})+\varepsilon_n,\quad\forall i=1,\ldots,n,
\end{equation}
where $\varepsilon_n>0$ is independent of the choice of $u^i$ and it tends to 0 as $n\to\infty$. Similar to the proof of Lemma \ref{lem:bound-u}, denote by $\widehat{P}^{u^i,(n)}=(\widehat{P}_t^{u^i,(n)})_{t\geq0}$ the price process with the policy $(u^i,u^{*,-i})=(u^{*,1},\ldots,u^{*,i-1},u^i,u^{*,i+1},\ldots,u^{*,n})$. That is, it satisfies the following SDE: for $\widehat{P}_0^{u^i,(n)}=p_0$, and
\begin{align}\label{eq:hatPn}
  d\widehat{P}_t^{u^i,(n)}&=\alpha\left[\beta-\frac{1}{n}\left(X_t^{i,u^i}+\sum_{j\neq i} X_t^{j,*}\right)-P_t^{*,(n)}\right]dt+\frac{\alpha }{n}\left(\gamma_iX_{t}^{i,*}dN_t^{i}+\sum_{j\neq i}\gamma_jX_{t}^{j,*}dN_t^{j}\right),
\end{align}
where $X^{i,*}=(X^{i,*}_t)_{t\geq0}$ is given by \eqref{eq:X*i}. The limiting  reward functional for firm $i$ (i.e., with the limiting price process $\bar{m}^P_t$ instead of $\widehat{P}^{u^i,(n)}$ in the reward function $R_i^{(n)}$) is defined by: for $u^i\in\mathbb{U}$,
{\small\begin{equation}\label{eq:limitbarRiu}
  \bar{R}_i(u^i)=\Ex\left[\int_{0}^{\infty}e^{-\rho t}\left((1-c_i)\bar{m}^P_t X_t^{i,u^i}-r_i(u_t^i)^2\right)dt\right].
\end{equation}}
Then, it follows from \eqref{eq:limitbarRiu} that
\begin{align}\label{eq:aim}
&~~~~\left|\sup_{u^i\in\widetilde{\mathbb{U}}^i}R_i^{(n)}(u^i,u^{*,-i})-R_i^{(n)}(u^{*,(n)})\right|\nonumber\\
&\leq \left|\sup_{u^i\in\widetilde{\mathbb{U}}^i}\left[R_i^{(n)}(u^i,u^{*,-i}) -\bar{R}_i(u^i)\right]\right|
+\left|\sup_{u^i\in\widetilde{\mathbb{U}}^i}\bar{R}_i(u^i)-R_i^{(n)}(u^{*,(n)})\right|
=: I_1^{(n)}+I_2^{(n)}.
\end{align}
First of all, it follows from H\"{o}lder's inequality that
\begin{align*}
  &\qquad R_i^{(n)}(u^i,u^{*,-i}) -\bar{R}_i(u^i)
  =(1-c_i)\Ex\left[\int_0^\infty e^{-\rho t}(\widehat{P}^{u^i,(n)}_t-\bar{m}^P_t)X_t^{i,u^i}dt\right]\nonumber\\
  &\leq \left\{\Ex\left[\int_0^\infty e^{-\rho t}|\widehat{P}^{u^i,(n)}_t-\bar{m}^P_t|^2dt\right]\|X^{i,u^i}\|_\rho^2\right\}^{1/2}
  \leq \left\{D_1(1+D_3)\Ex\left[\int_0^\infty e^{-\rho t}|\widehat{P}^{u^i,(n)}_t-\bar{m}^P_t|^2dt\right]\right\}^{1/2},
\end{align*}
where the last inequality is due to Lemma \ref{lem:estimates} and Lemma \ref{lem:bound-u}. We also note that
\begin{align*}
  \Ex\left[\int_0^\infty e^{-\rho t}|\widehat{P}^{u^i,(n)}_t-\bar{m}^P_t|^2dt\right]
  \leq 2\Ex\left[\int_0^\infty e^{-\rho t}|\widehat{P}^{u^i,(n)}_t-P_t^{*,(n)}|^2dt\right]+2\Ex\left[\int_0^\infty e^{-\rho t}|P_t^{*,(n)}-\bar{m}^P_t|^2dt\right].
\end{align*}
Using \eqref{eq:hatPn} and \eqref{eq:Pn*}, it follows from a similar argument used in the proof of Lemma \ref{lem:es-L2P+mean} that
\begin{align*}
  \sup_{u^i\in\widetilde{\mathbb{U}}^i}\Ex\left[\left|\widehat{P}^{u^i,(n)}_t-P_t^{*,(n)}\right|^2\right]\leq \sup_{u^i\in\widetilde{\mathbb{U}}^i}\frac{C}{n^2}\Ex\left[\left|X_t^{i,u^i}-X_t^{i,*}\right|^2\right]\leq \frac{C}{n^2},
\end{align*}
where $C>0$ is a constant independent of $i$. Then, Lemma \ref{lem:cov-L2P+mean} yields that
\begin{align*}
\sup_{u^i\in\widetilde{\mathbb{U}}^i}\Ex\left[\int_0^\infty e^{-\rho t}|\widehat{P}^{u^i,(n)}_t-\bar{m}^P_t|^2dt\right]\to 0,\quad n\to\infty.
\end{align*}
Thus, $I_1^{(n)}\to 0$ as $n\to\infty$. Similarly to the proof of Proposition \ref{prop:repre-optimal}, we can show that $\sup_{u^i\in\widetilde{\mathbb{U}}^i}\bar{R}_i(u^i)=\bar{R}_i(u^{*,i})$.
Therefore
{\small\begin{align*}
  (I_2^{(n)})^2&=\left|\bar{R}_i(u^{*,i})-R_i^{(n)}(u^{*,(n)})\right|^2
  \leq \Ex\left[\int_0^\infty e^{-\rho t}|\bar{m}^P_t-P_t^{*,(n)}|^2dt\right]\Ex\left[\int_0^\infty e^{-\rho t}|X_t^{i,*}|^2dt\right]\\
  &\leq C\Ex\left[\int_0^\infty e^{-\rho t}|\bar{m}^P_t-P_t^{*,{(n)}}|^2dt\right]\to 0,\qquad \textrm{as}~n\to\infty.
\end{align*}}
Then, the desired result follows from \eqref{eq:aim} with $\varepsilon_n\leq I_1^{(n)}+I_2^{(n)}\to0$ as $n\to\infty$. 
\end{proof}

\end{document}